\documentclass[10pt]{amsart}
\usepackage{amssymb}
\usepackage{graphicx, pinlabel}
\usepackage{graphics}
\usepackage{amscd}
\usepackage{enumerate}
\usepackage{comment}
\usepackage[all]{xy}
\usepackage{mathtools} %needed for floor and ceiling
\usepackage{color}
\usepackage[table]{xcolor}
\usepackage{xspace} %for spaces after Spinc
\usepackage[margin=3cm, marginparwidth=2.5cm]{geometry}
\usepackage{hyperref}

\newtheorem{thm}{Theorem}[section]
\newtheorem{theorem}[thm]{Theorem}

\newtheorem{conjecture}[thm]{Conjecture}

\theoremstyle{definition}

\newtheorem{question}[thm]{Question}
\theoremstyle{remark}

\newtheorem{remark}[thm]{Remark}

\numberwithin{equation}{section}

\newcommand{\mc}{\mathcal}

\renewcommand{\theta}{\vartheta}

\newcommand{\N}{\mathbb{N}}
\newcommand{\Z}{\mathbb{Z}}
\newcommand{\Q}{\mathbb{Q}}

\DeclareMathOperator{\Kh}{Kh}

\begin{document}

\title[The Knight Move Conjecture is false]{The Knight Move Conjecture is false}%

\author{Ciprian Manolescu}%
\address{Department of Mathematics, UCLA,
520 Portola Plaza, Los Angeles, CA 90095}%
\email{cm@math.ucla.edu}%
%\thanks{}

\author{Marco Marengon}%
\address{Department of Mathematics, UCLA,
520 Portola Plaza, Los Angeles, CA 90095}%
\email{marengon@ucla.edu}%
%\thanks{}

%\subjclass[2010]{}%
%\keywords{}

\date{}%
%\dedicatory{}%
%\commby{}%

% ----------------------------------------------------------------
\begin{abstract}
The Knight Move Conjecture claims that the Khovanov homology of any knot decomposes as direct sums of some ``knight move'' pairs and a single ``pawn move'' pair. This is true for instance whenever the Lee spectral sequence from Khovanov homology to $\Q^2$ converges on the second page, as it does for all alternating knots and knots with unknotting number at most $2$. We present a counterexample to the Knight Move Conjecture. For this knot, the Lee spectral sequence admits a nontrivial differential of bidegree $(1,8)$.
\end{abstract}

\maketitle
% ----------------------------------------------------------------

\section{Introduction}

Almost 20 years ago, Khovanov~\cite{Kh} introduced a categorification of the Jones polynomial, now known by the name of \emph{Khovanov homology}. This is an invariant of links in $S^3$ that is strictly more powerful than the Jones polynomial~\cite{BN}, and it detects the unknot~\cite{KM}. Furthermore, using Khovanov homology, Rasmussen defined a concordance homomorphism $s \colon \mc C \to 2\Z$ from the smooth knot concordance group, and used it to give the first combinatorial proof of Milnor's conjecture~\cite{s-invariant}.

Given a knot $K \subset S^3$, its Khovanov homology over $\Q$ is a bigraded vector space over $\Q$, endowed with a \emph{homological grading} $i \in \Z$ and a \emph{quantum grading} $j \in 1 + 2\Z$. We denote this bigrading by $(i,j)$. We denote the Khovanov homology of a knot $K \subset S^3$ by
\[
\Kh(K) = \bigoplus_{\substack{i \in \Z \\  j \in 1+2\Z }} \Kh^{i,j}(K),
\]
and its Poincar\'e series by $\Kh(K)(t,q) = \sum_{i,j} \dim_{\Q}\left(\Kh^{i,j}(L)\right) t^i q^j$.

\subsection{The structure of Khovanov homology}
An early conjecture about the structure of Khovanov homology~\cite[Conjecture 1]{BN}, known as the \emph{Knight Move Conjecture}, is due to Bar-Natan, Garoufalidis, and Khovanov. It says that it is always possible to decompose the Khovanov homology of a knot into the direct sum of elementary pieces.

\begin{conjecture}[Knight Move Conjecture~\cite{BN}]
Given a knot $K$, its Khovanov homology over $\Q$ is the direct sum of a single \emph{pawn move} piece
\[
\Q\{0,s-1\} \oplus \Q\{0,s+1\},
\]
where $s$ is Rasmussen's invariant, and several \emph{knight move} pieces
\[
\Q\{i, j\} \oplus \Q\{i+1, j+4\},
\]
for various $i,j \in \Z$.
\end{conjecture}

In terms of Poincar\'e series, this conjecture can be rewritten as follows (see \cite[Conjecture 5.2]{mnm}):

\begin{conjecture}[Reformulation of the Knight Move Conjecture]
For any knot $K$, there is a Laurent polynomial $f_2 \in \N[t^{\pm1}, q^{\pm1}]$ so that
\[
\Kh(K)(t,q) = q^s (q+q^{-1}) +  f_{2}(t,q) (1 + tq^{4}).
\]
\end{conjecture}

\subsection{Lee's deformation}
In \cite{Lee}, Lee introduced a deformation of the (co-)chain complex yielding Khovanov homology, which in fact is a filtered differential, where the filtration level is given by the quantum degree. The zeroth page of the resulting spectral sequence is the usual Khovanov complex, with the usual differential $d_0$. Thus, the first page $E_1$ is simply Khovanov homology.\footnote{In some of the literature, for example in \cite{s-invariant}, what we call the $E_n$ page of the Lee spectral sequence is denoted $E_{n+1}$, and our differential $d_n$ is their $d_{n+1}$.} The higher differentials $d_n$ on $E_n$ have degree $(1, 4n)$. Lastly, if $K$ is a knot, the resulting spectral sequence converges to $\Q\{0,s-1\} \oplus \Q\{0,s+1\}$, where $s$ is Rasmussen's invariant.

Using the above properties, it is immediate to check that if the Lee spectral sequence of a knot $K$ degenerates after the first page, then there must be a Knight Move decomposition of $\Kh(K)$. This is true for example for all alternating knots~\cite{LeeAlt}, and more generally for all quasi-alternating knots~\cite{MO}, as well as for all knots with unknotting number not bigger than 2~\cite{AD}.

For a general knot, a corollary of Lee's spectral sequence is that we have a decomposition of Khovanov homology into a pawn move and several, possibly ``longer'' knight moves, of the form
\[
\Q\{i, j\} \oplus \Q\{i+1, j+4n\}.
\]
In other words, for any knot $K$, there is a family of two variable Laurent polynomials $f_{2l} \in \N[t^{\pm1}, q^{\pm1}]$, for $l \geq 1$, so that
\[
\Kh(K)(t,q) = q^s (q+q^{-1}) + \sum_{l \geq 1} f_{2l}(t,q) (1 + tq^{4l}).
\]
The Knight Move Conjecture is equivalent to saying that $f_{2l}$ can be set to $0$ for all $l \geq 2$.

In this note, we present a counterexample to the Knight Move Conjecture. The example that we give has a non-trivial differential $d_2$.

\section{The counterexample}

Our counterexample is the knot $K$ illustrated in Figure \ref{fig:K}. It is obtained from an $8$-crossing diagram of the unknot by doing a full positive twist along $6$ strands. The resulting diagram has $38$ crossings.

\begin{center}
\begin{figure}
\resizebox{.75\textwidth}{!}{
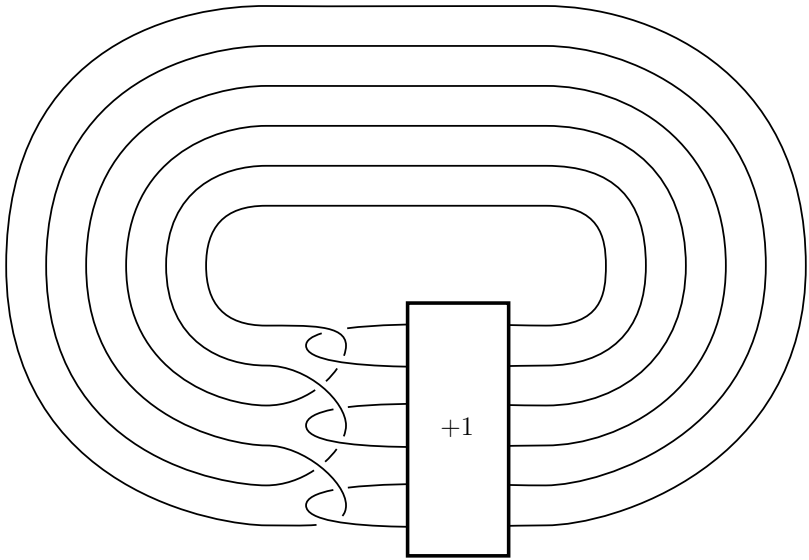
}
\caption{The knot $K$. The box labelled $+1$ denotes a full positive twist.}
\label{fig:K}
\end{figure}
\end{center}

\begin{theorem}
The knot $K$ in Figure \ref{fig:K} does not satisfy the Knight Move Conjecture. Moreover, the second Lee differential $d_2$ of bidegree $(1,8)$ is non-vanishing.
\end{theorem}
\begin{proof}
The Khovanov homology of $K$ is computed using the program ``JavaKh-v2'', an update by Scott Morrison of Jeremy Green's original program, both of which are available on the Knot Atlas~\cite{KAT}. The result is shown in Table \ref{table}. The entry $tq$ (marked in red) is non-empty. If the Knight Move Conjecture were true, this should be matched by a non-zero entry in either $q^{-3}$ or $t^2q^5$. However, these are both empty.

Regarding the Lee spectral sequence, the entry $tq$ cannot be cancelled by a $d_1$ differential, because both the entries $q^{-3}$ and $t^2q^5$ are empty. It follows that it must be cancelled by a higher differential, which is necessarily $d_2$, since there is no room for non-trivial maps of bidegree $(1, 4n)$ for $n \geq 3$, as one can easily check from Table \ref{table}.
\end{proof}

{\small
\begin{center}
\begin{table}
  \begin{tabular}{ r|| c | c | c | c | c | c | c | c | c | c | c | c | c | c | c | c | c | c | c | c | c | c | c }
									 &-18&-17&-16&-15&-14&-13&-12&-11&-10&-9 &-8 &-7 &-6 &-5 &-4 &-3 &-2 &-1 & 0 & 1 & 2 & 3 & 4 \\
									\hline\hline
									13 &   &   &   &   &   &   &   &   &   &   &   &   &   &   &   &   &   &   &   &   &   &   & 1 \\ \hline
									11 &   &   &   &   &   &   &   &   &   &   &   &   &   &   &   &   &   &   &   &   &   & 1 &   \\ \hline
									9	 &   &   &   &   &   &   &   &   &   &   &   &   &   &   &   &   &   &   &   & 1 & 2 & 1 &   \\ \hline
									7	 &   &   &   &   &   &   &   &   &   &   &   &   &   &   &   &   &   &   & 3 & 4 & 2 &   &   \\ \hline
									5	 &   &   &   &   &   &   &   &   &   &   &   &   &   &   &   &   &   & 5 & 4 & 1 &   &   &   \\ \hline
									3	 &   &   &   &   &   &   &   &   &   &   &   &   &   &   &   & 1 & 6 & 6 & 4 & 1 &   &   &   \\ \hline
									1	 &   &   &   &   &   &   &   &   &   &   &   &   &   &   & 3 & 9 &10 & 4 & 1 & 1\cellcolor{red} &   &   &   \\ \hline
									-1 &   &   &   &   &   &   &   &   &   &   &   &   &   & 3 & 9 & 8 & 3 & 1 & 1 &   &   &   &   \\ \hline
									-3 &   &   &   &   &   &   &   &   &   &   &   &   & 3 &10 &12 & 6 & 1 &   &   &   &   &   &   \\ \hline
									-5 &   &   &   &   &   &   &   &   &   &   & 1 & 5 &10 &10 & 2 &   & 1 &   &   &   &   &   &   \\ \hline
									-7 &   &   &   &   &   &   &   &   & 1 & 2 & 4 & 6 & 7 & 3 & 1 &   &   &   &   &   &   &   &   \\ \hline
									-9 &   &   &   &   &   &   &   & 2 & 1 & 3 & 7 & 8 & 2 &   &   &   &   &   &   &   &   &   &   \\ \hline
									-11&   &   &   &   &   &   & 2 & 2 & 4 & 5 & 3 &   &   &   &   &   &   &   &   &   &   &   &   \\ \hline
									-13&   &   &   &   &   & 3 & 4 & 2 & 2 & 2 & 1 &   &   &   &   &   &   &   &   &   &   &   &   \\ \hline
									-15&   &   &   &   & 3 & 3 & 1 & 2 & 1 &   &   &   &   &   &   &   &   &   &   &   &   &   &   \\ \hline
									-17&   &   &   & 2 & 3 & 2 & 1 &   &   &   &   &   &   &   &   &   &   &   &   &   &   &   &   \\ \hline
									-19&   &   & 2 & 3 & 1 &   &   &   &   &   &   &   &   &   &   &   &   &   &   &   &   &   &   \\ \hline
									-21&   & 1 & 2 &   &   &   &   &   &   &   &   &   &   &   &   &   &   &   &   &   &   &   &   \\ \hline
									-23&   & 2 &   &   &   &   &   &   &   &   &   &   &   &   &   &   &   &   &   &   &   &   &   \\ \hline
									-25& 1 &   &   &   &   &   &   &   &   &   &   &   &   &   &   &   &   &   &   &   &   &   &   \\
  \end{tabular}
	\vspace{10pt}
\caption{The Khovanov homology of the knot $K$. The homological grading $i$ is on the horizontal axis, and the quantum grading $j$ on the vertical axis. The red box marks an entry that cannot be cancelled by a $d_1$ differential.}
\label{table}
\end{table}
\end{center}
}

In fact, one can determine the whole structure of the Lee spectral sequence for $K$ using the program ``UniversalKh'' of Scott Morrison \cite{KAT, UniversalKh}. It turns out that the $d_1$ differential (the knight move) cancels most of the terms in Khovanov homology, leaving only four copies of $\Q$ on the $E_2$ page, in bidegrees $(0, -1)$, $(0, 1)$, $(1, 1)$ and $(2, 9)$. The last two are cancelled by the $d_2$ differential, and the first two survive to the $E_{\infty}$ page. Rasmussen's invariant for this knot is $s=0$.

\begin{remark}
We came across the knot $K$ while studying the {\em generalized crossing changes} introduced by Cochran and Tweedy in \cite{positive}. The full twist shown in Figure \ref{fig:K} is an example of a generalized negative crossing. The resulting knot $K$ is slice in the blown-up ball $B^4 \# \overline{\mathbb{CP}}^2$, but it is not slice in $B^4$, because its Alexander polynomial
$$ \Delta_K(t) = -3t^{-1} + 7 - 3t$$
does not satisfy the Fox-Milnor criterion.
\end{remark}

\begin{remark}
It is easy to see that the knot $K$ can be unknotted by three crossing changes. Since the Knight Move Conjecture holds for knots of unknotting number at most two~\cite{AD}, it follows that $K$ has 
unknotting number $3$.
\end{remark}

We end with an open problem.

\begin{question}
Given any $n \geq 3$, does there exist a knot for which the $d_n$ differential in the Lee spectral sequence is nonzero?
\end{question}

In view of the work of Alishahi and Dowlin \cite{AD}, if for a knot $K$ we have $d_n \neq 0$, then $K$ needs to have unknotting number at least $2n-1$.

\subsection*{Acknowledgements}
We are grateful to Nate Dowlin, Lukas Lewark, Scott Morrison, Jake Rasmussen and Alexander Shumakovitch for helpful comments on the paper.

The first author was partially supported by the NSF grant DMS-1708320.

% ----------------------------------------------------------------
\bibliographystyle{amsplain}
\bibliography{topology}

\providecommand{\bysame}{\leavevmode\hbox to3em{\hrulefill}\thinspace}
\providecommand{\MR}{\relax\ifhmode\unskip\space\fi MR }
\providecommand{\MRhref}[2]{%
  \href{http://www.ams.org/mathscinet-getitem?mr=#1}{#2}
}
\providecommand{\href}[2]{#2}
\begin{thebibliography}{10}

\bibitem{KAT}
\emph{The {K}not {A}tlas}, \url{http://www.katlas.org}.

\bibitem{AD}
Akram Alishahi and Nathan Dowlin, \emph{The {L}ee spectral sequence, unknotting
  number, and the {K}night {M}ove {C}onjecture}, preprint,
  \url{arXiv:1710.07875} (2017).

\bibitem{BN}
Dror Bar-Natan, \emph{On {K}hovanov's categorification of the {J}ones
  polynomial}, Algebraic \& Geometric Topology \textbf{2} (2002), no.~1,
  337--370.

\bibitem{positive}
Tim~D. Cochran and Eamonn Tweedy, \emph{Positive links}, Algebraic \& Geometric
  Topology \textbf{14} (2014), no.~4, 2259--2298.

\bibitem{mnm}
Michael~H. Freedman, Robert~E. Gompf, Scott Morrison, and Kevin Walker,
  \emph{Man and machine thinking about the smooth 4-dimensional {P}oincar{\'e}
  conjecture}, Quantum Topology \textbf{1} (2010), no.~2, 171--208.

\bibitem{Kh}
Mikhail Khovanov, \emph{A categorification of the {J}ones polynomial}, Duke
  Math. J. \textbf{101} (2000), no.~3, 359--426.

\bibitem{KM}
Peter~B. Kronheimer and Tomasz~S. Mrowka, \emph{Khovanov homology is an
  unknot-detector}, Publications {M}ath{\'e}matiques de l'IH{\'E}S \textbf{113}
  (2011), no.~1, 97--208.

\bibitem{LeeAlt}
Eun~Soo Lee, \emph{The support of the {K}hovanov's invariants for alternating
  knots}, preprint, \url{arXiv:math/0201105} (2002).

\bibitem{Lee}
\bysame, \emph{An endomorphism of the {K}hovanov invariant}, Advances in
  Mathematics \textbf{197} (2005), no.~2, 554--586.

\bibitem{MO}
Ciprian Manolescu and Peter Ozsv{\'a}th, \emph{On the {K}hovanov and knot
  {F}loer homologies of quasi-alternating links}, Proceedings of the G\"okova
  Geometry-Topology Conference 2007 (2008), 60--81.

\bibitem{UniversalKh}
Scott Morrison, \emph{Genus bounds and spectral sequences made easy}, slides
  for a talk in Kyoto (2007),
  \url{https://tqft.net/math/FreeAlphaKyoto2007.pdf}.

\bibitem{s-invariant}
Jacob Rasmussen, \emph{Khovanov homology and the slice genus}, Inventiones
  Mathematicae \textbf{182} (2010), no.~2, 419--447.

\end{thebibliography}
\end{document}